\documentclass[12pt]{amsart}%
\usepackage{amsfonts}
\usepackage{graphicx}
\usepackage{amsmath}
\usepackage{amssymb}%
\usepackage{color}
\newtheorem{theorem}{Theorem}

\newtheorem{lemma}{Lemma}

\newtheorem{remark}{Remark}

\theoremstyle{remark}
\newtheorem{example}{\bf Example}

\input epsf
\begin{document}

\title[Jumps in spectral networks]{Combinatorial description of jumps in spectral networks}
\author{Anastasia Frolova and Alexander Vasil'ev}

\address{Department of Mathematics,
University of Bergen, P.O.~Box~7800, Bergen N-5020, Norway}

\email{Anastasia.Frolova@uib.no} \email{Alexander.Vasiliev@math.uib.no}

\thanks{ The authors have been  supported by the grants of the Norwegian Research Council \#239033/F20,
\#213440/BG; and EU FP7 IRSES program STREVCOMS, grant  no.
PIRSES-GA-2013-612669. }


\subjclass[2010]{Primary 58K20; Secondary 30F30, 52B11, 58K15, 81T40, 81T60}

\keywords{Spectral network, quadratic differential, Stokes line, weighted diagram, Stasheff polytope}


\begin{abstract}
We describe a graph parametrization of rational quadratic differentials with presence of a simple pole, whose critical trajectories form a network depending on parameters  focusing on the network topological jumps. Obtained bifurcation diagrams are associated with the Stasheff polytopes.
\end{abstract}
\maketitle

\begin{section}{Introduction}

The problem of BPS (Bogomol'nyi--Prasad--Sommerfield) wall crossing have received much attention the last decade, see e.g, \cite{Cecotti, Cecotti2, Cecotti3, Gaiotto, Joyce, Kontsevich}. In physics terms, a supersymmetric particle may change from stable to unstable crossing loci (walls) in a parameter space. Considering four-dimensional $\mathcal N=2$ theories coupled to surface defects, particularly the theories of class $S$, see \cite{Witten},  Gaiotto, Moore, and Neitzke \cite{Gaiotto2} introduced {\it spectral networks} of trajectories on Riemann surfaces obeying certain local rules aiming at the characterization of the possible spectra of BPS states and their allowed changes under continuous deformations of the theory. Given a compact Riemann surface $\mathcal R$ with punctures and a Lie algebra $\mathfrak g$ of ADE type, e.g., SU$(2)$ in our case, there exists a corresponding four-dimensional quantum field theory $S(\mathcal R, \mathfrak g)$, see \cite{Joyce, Witten}. The spectral network is defined by the critical trajectories of a quadratic differential $q$ given by
$q(z)dz^2$ in a local parameter $z$, which
defines a singular measured foliation of $\mathcal R$ with singularities at the zeros and poles of $q$. The differential is holomorphic on $\mathcal R$ and has possible poles at the punctures. The trajectories emerging from the zeros form the spectral network. For certain values of the zeros, there occur critical trajectories starting and ending at them, and 
we say that the network undergoes jumps and splits $\mathcal R$ into cells. Generic small variation of zeros changes the network by isotopy whereas the jumps occur
for certain values of them. Such critical trajectories we will call {\it short}. Counting the special trajectories is related to generalized Donaldson-Thomas invariants of the theory.

Short trajectories of $q$ turn to play an important role also in potential theory, approximation theory and other branches of mathematics. For example, short trajectories of rational quadratic differentials describe limiting distributions of certain types of orthogonal polynomials,  see e.g., \cite{Martinez2011, Martinez2014, ShapiroTakemura2011}. Motivated by applications to minimal surfaces, Bruce and O'Shea  published a preprint \cite{Bruce1995}, where the short trajectories characterized  umbilical points  and the geometry of unfolding.
Baryshnikov  \cite{Baryshnikov1997, Baryshnikov2001} described the combinatorial structure of the Stokes sets for polynomials in one variable  by bifurcation diagrams, 
and in particular, encoded the short trajectories of the differential $q$ in the simplest case when $\mathcal R=\mathbb C$ and $q$ is a versal deformation
of $z^ndz^2$. 
It was proved in \cite{Kostov1993} that the versal deformation of 
$z^n\,dz^2$ is 
the family $(z^n+a_{n-2} z^{n-2}+...+a_0)\,dz^2,$ where $a_k,$ $0\leq k\leq n-2,$ are complex parameters. It can be understood as a family which includes in a certain sense all  quadratic differentials of the form $p_n(z)\,dz^2,$ where $p_n(z)$ is a monic polynomial of degree $n$. The set of all parameters $(a_{n-2},...,a_0)$ in the parameter base space $\mathbb{C}^{n-2}$, for which the corresponding quadratic differential has a short trajectory, is the bifurcation diagram of the versal deformation, i.e., whenever a parameter $(a_{n-2},...,a_0)$ belongs to the bifurcation diagram, a small change of parameter 
causes a significant change in the trajectory structure. Using formal power series Bruce and O'Shea gave an explicit form of the bifurcation diagram for the case $n=2.$ 
They initiated the study of combinatorial structure of  bifurcation diagrams for arbitrary $n,$ which was completed by Baryshnikov \cite{Baryshnikov1997, Baryshnikov2001}, who gave combinatorial and  geometric descriptions of the set of polynomial quadratic differentials with  short trajectories. 
He also established correspondence between polynomial quadratic differentials and weighted graphs, and used 
the connection between weighted graphs and the Stasheff polyhedra.

The latter and physics motivation encouraged us to consider the case of quadratic differential with the presence of poles, in particular, the case of one simple
pole. The domains in the trajectory structure of the differential in our approach contain ending domains (or half-planes) and strip domains. More poles destroy completely the proposed picture because even two simple poles guarantee new types of domains, i.e., ring domains and dense structures.  We so far do not know what kind of graphs could parametrize them.
So our result in some sense extends Baryshnikov's approach up to the end.

Let us remark that  different graph encodings of quadratic differentials
were also used as a tool for solving a number of other problems.
For example, Bogatyr{\"e}v in \cite{Bogatyrev2003} used certain graphs based on quadratic differentials
in connection with the problem of description of extremal polynomials. Solynin \cite{Solynin2009} established the connection between weighted graphs and quadratic differentials with closed trajectories.

The outline of the paper is as follows.
In Section \ref{sbar},
we introduce 
the correspondence between weighted chord diagrams and the Stasheff polyhedra through the balanced weights following
\cite{Baryshnikov1997}. 
In Section \ref{sqd}, we discuss briefly the trajectory structure of rational quadratic differentials with a simple pole.
We establish one-to-one correspondence between weighted graphs and rational quadratic differentials with a simple pole in Section \ref{graphs}. Graphs and weighted chord diagrams identified with the quadratic differentials with short trajectories are described there.
The latter allows us to use the correspondence between weighted chord diagrams and the Stasheff polyhedra to obtain  an analogue of the bifurcation diagram for the case of rational quadratic differentials with a simple pole.
\vspace{10pt}

\noindent
{\bf Acknowledgement.}  The authors acknowledge many helpful discussions with prof. Boris Shapiro (Stockholm University) and the Mittag-Leffler Institute where this study started.

\end{section}

\begin{section}{Weighted chord diagrams and balanced weights }\label{sbar}

Following Baryshnikov \cite{Baryshnikov1997, Baryshnikov2001} we introduce weighted chord diagrams, Stasheff polyhedra, balanced weights, and describe  the correspondence between them.

A polytope $C$ in $\mathbb{R}^d$ is a convex hull of a certain number of points in $\mathbb{R}^d.$ 
If $C$ intersects a hyperplane $H$ and lies entirely in one of the half-spaces defined by $H$, we call $H\cap C$ a face of $C$. The vertices and edges of a polytope $C$ are $0-$ and $1-$dimentional faces of $C$ respectively. 
Any given vector $v\in\mathbb{R}^d$ determines a face $F_{v}(C)$ of $C$:
\[ F_{v}(C)=\{ x\in C : x\cdot v\geq y \cdot v\,\,\, \forall y\in C\}.
\]
$F_{v}(C)$ is an intersection of $C$ with a hyperplane
which  goes through the point $\mathrm{arg max}_{x\in C}\, x\cdot v$ and has $v$ as the normal vector. For $v=0$ we obtain the entire polytope $C$.
For any face $F$ of $C$ we define a normal cone $N_{F}(C)$ as
\[ N_{F}(C)=\{v\in\mathbb{R}^d : F=F_{v}(C)\}.
\]
Note that if the face $F$ has dimension $l$ and $l\leq d$, then the dimension of the normal cone $N_{F}(C)$ is $d-l$.
The collection of all normal cones of $C$ is called the normal fan of $C$.

Stasheff polyhedron (associahedron) $K_n$ is an $(n-2)-$dimensional polytope. Each vertex  of $K_n$ corresponds to a bracketing of a string of $n$ symbols, and each edge corresponds to a single application of associativity rule.
For example, $K_3$ consists of two vertices  represented by $(ab)c$ and $a(bc)$ and one edge connecting them.  Analogously, $K_4$ is a pentagon and $K_5$ is a polyhedron.

Alternatively, $K_n$ can be realized as a
polytope whose vertices represent triangulations of a regular $(n+1)-$gon and edges represent diagonal flips.
Triangulation of a polygon is a collection of non-intersecting diagonals; it is said to be incomplete if the number of diagonals is not maximal. The vertices of the polytope dual to $K_n$ correspond to  incomplete triangulations of the $(n+1)-$gon.

\begin{example}
The triangulation realization of $K_4$ is shown on figure \ref{k4}.
\begin{figure}
\includegraphics[scale=0.2]{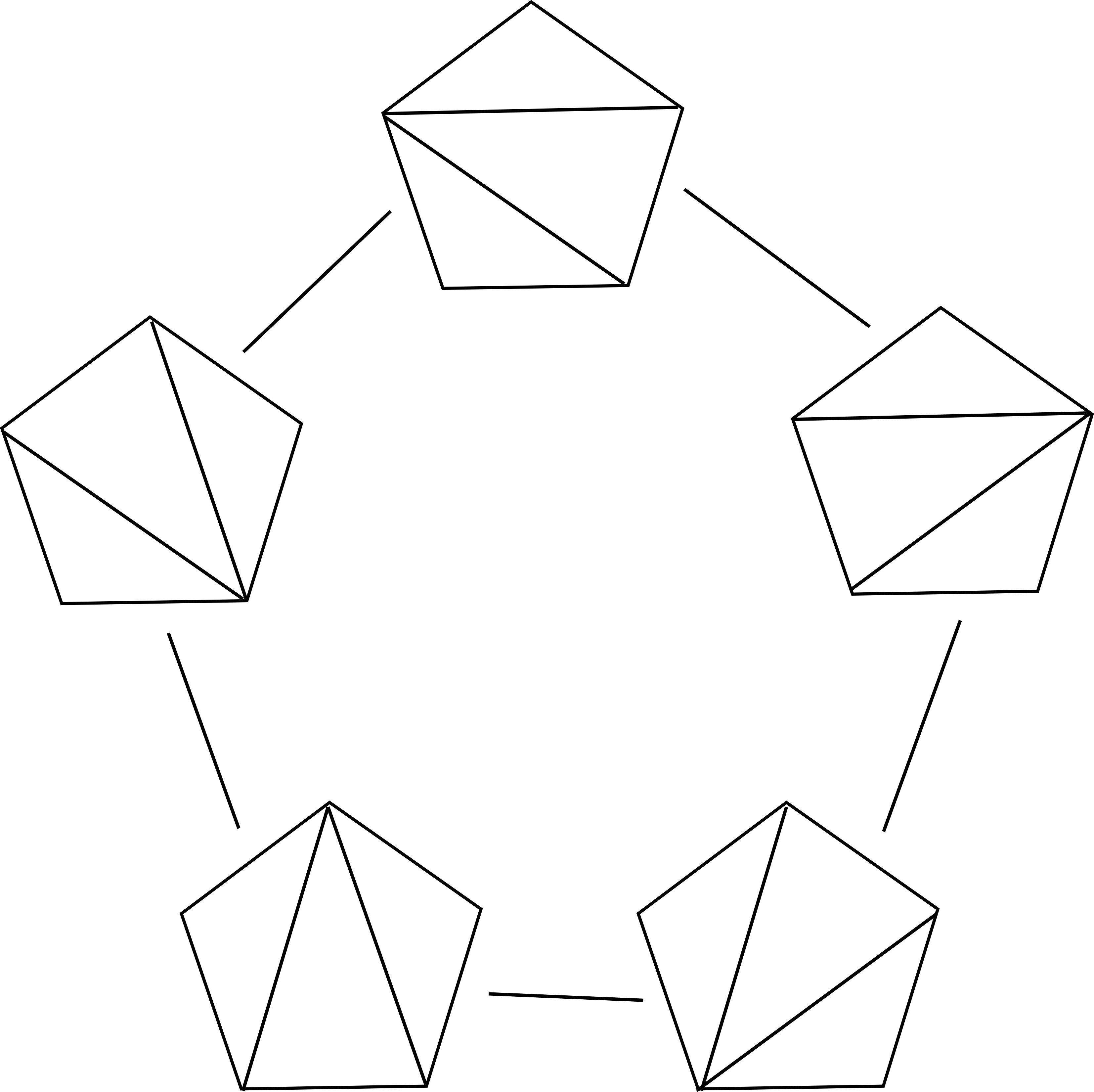}\label{k4}
\caption{The triangulation realization of $K_4.$}
\end{figure}
\end{example}

Normal fan $\Sigma_{n}$ to the Stasheff polytope $K_n$ is called the Stasheff fan.
The union of the cones of $\Sigma_{n}$ constitutes $\mathbb{R}^{n-2}.$ The number of full-dimensional cones is equal to the Catalan number $c_{n-1}=\frac{1}{n}  \binom {2n-2} {n-1}.$
 
\begin{example}
The Stasheff fan $\Sigma_4$ is illustrated on figure \ref{stash1}.

\begin{figure}\label{stash1}
\includegraphics[scale=0.2]{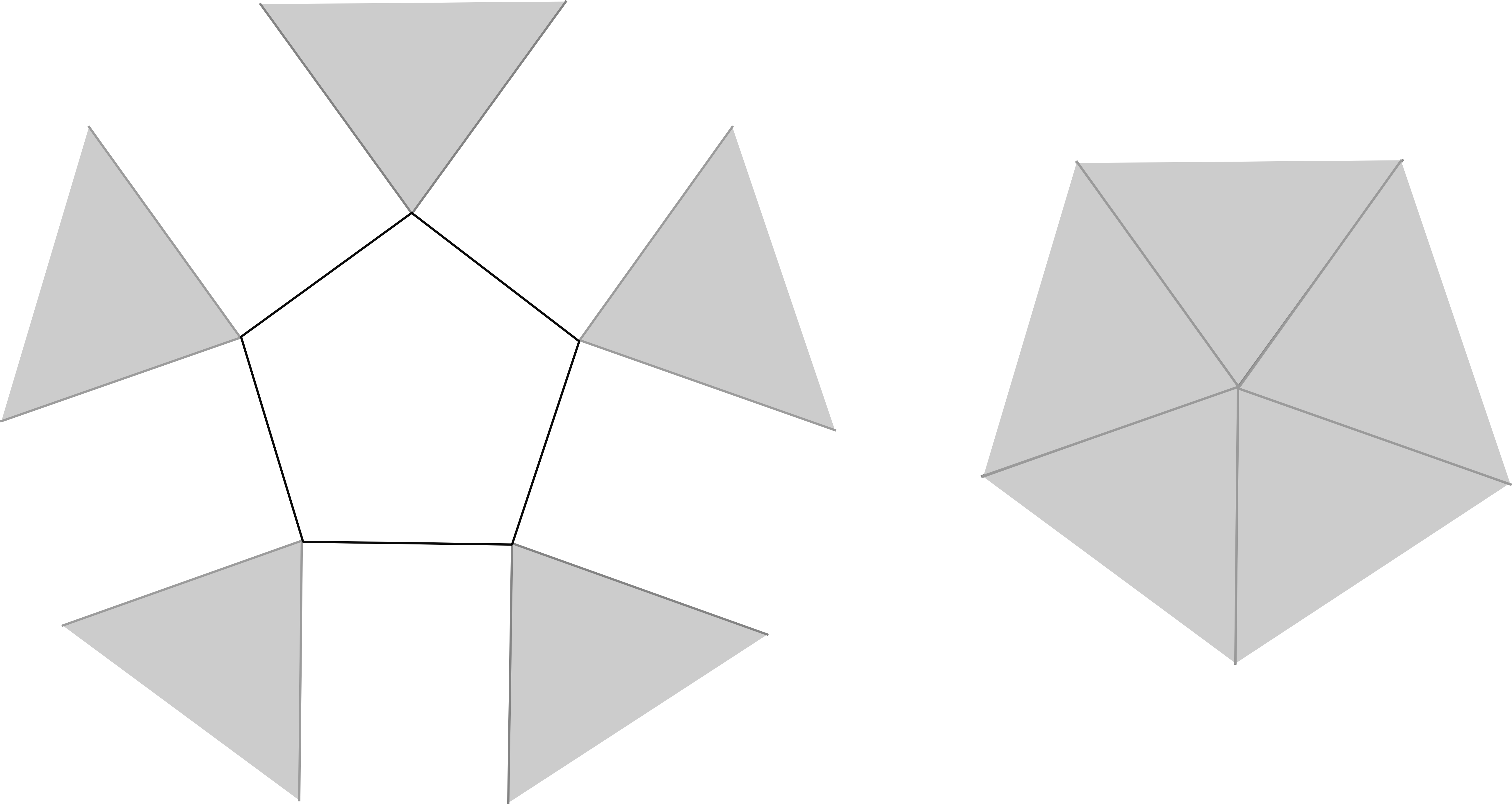}\caption{Normal fan $\Sigma_4$}
\end{figure}
\end{example}

Suppose we have a convex regular $(n+1)-$gon $P$.
	$P$ together with some weighted non-intersecting diagonals is called a weighted chord diagram. In this case we say that the weighted chord diagram is based on $P.$
	
	 A balanced weight is a function defined on the vertices of the $(n+1)-$gon $P$,
such that the sum of its values at the vertices is zero and the geometric center of masses is at the origin. A balanced weight $f$ is called degenerate if there exists a real linear function $L$ 
and vertices $a_{1,2,3,4},$ such that 
$L$ majorizes $f$ and coincides with it at the vertices $a_{1,2,3,4}.
$

Balanced weights form a linear space of real dimension $n-2$. According to  Baryshnikov, the degenerate balanced weights form a fan $\Sigma_{n}$, which is a normal fan for the Stasheff polytope $K_n$.

In what follows, we describe the correspondence between weighted chord diagrams and balanced weights.
\begin{lemma}
There is one-to-one correspondence between balanced weights and 
weighted chord diagrams.
\end{lemma}
\begin{proof}
Let us show that each balanced weight gives rise to a weighted chord diagram.
We fix a point $p$ lying on the plane of the polygon $P$ and in general position with respect to $P$. Let $a$ and $c$ be two arbitrary non-adjacent vertices of $P$. We consider all the real linear functions $L$ on the plane of $P$, such that
$L(a)=f(a),$ $L(c)=f(c)$, and  $\left.L\right|_P(z)\ge f(z)$ for any vertex $z$ of $P$. The values of such linear functions at $p$ swipe out an interval of length $v$, $v\geq 0$. If $v> 0$, we construct a diagonal with weight $v$ joining $a$ and $c$. We go through this procedure for any pair of non-adjacent vertices
of $P$ we construct all possible diagonals. Construction of a chord is illustrated on figure \ref{chord}.

Note that the diagonals in the resulting diagram do not intersect, i.e., we obtain a weighted chord diagram. Suppose we have constructed two intersecting diagonals $a_1a_2$ and $b_1b_2$.
Then there exist linear functions $L$ and $\Lambda,$ satisfying the relations
\begin{equation}\label{lin}
\begin{array}{ll}
  \left.L\right|_{P}\geq f,\,\,  L(a_{1,2})=f(a_{1,2}),\\ \\
  \left.\Lambda\right|_{P}\geq f,\,\,\Lambda(b_{1,2})=f(b_{1,2}).
  \end{array}
\end{equation}

 The latter gives us that
 \[
 \begin{array}{ll}
  L(a_{1,2})=f(a_{1,2})\leq \Lambda(a_{1,2}),\\
  \Lambda(b_{1,2})=f(b_{1,2})\leq L(b_{1,2}).
  \end{array}
\]
 Thus, the function $d(z)=\Lambda(z)-L(z)$ satisfies the inequalities $d(a_{1,2})\geq 0$ and 
$d(b_{1,2})\leq 0.$ As $a_1$ and $a_2$ lie on different sides with respect to $b_1b_2$ and $d$ is real linear, we obtain that $d$ vanishes identically. Therefore, the diagonals $a_1a_2$ and $b_1b_2$ fail to exist and we arrive at a contradiction.
\begin{figure}
\includegraphics[scale=0.6]{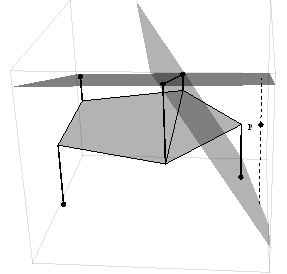}\caption{Construction of a chord.}\label{chord}
\end{figure}

Analogously, for each weighted chord diagram there is a balanced weight corresponding to it.
\end{proof}
Furthermore,  there is a one-to-one correspondence between the degenerate balanced weights and weighted
chord diagrams with incomplete triangulation. We separate the proof of this fact into three lemmas.

\begin{lemma}\label{oneside}
Suppose $f$ is a degenerate balanced weight and $P$ is the weighted
chord diagram corresponding to $f$. Then $P$ has an incomplete triangulation.
\end{lemma}

\begin{proof}

Suppose we are given a degenerate weight $f,$ i.e., there exists a linear function $L$ and vertices $a_k\in P,$ $1\leq k \leq 4,$ 
such that $\left.L\right|_P\geq f,$ $L(a_k)=f(a_k), $ $1\leq k\leq 4.$ The chord diagram corresponding to the weight $f$ can not have
a diagonal that intersects the interior of the quadrilateral $Q$ formed by $a_k,$ $1\leq k \leq 4,$ and thus, triangulation of $P$ is not complete. 
To show this we assume that $P$ has non-adjacent vertices $b_1$ and $b_2,$ such that the diagonal $b_1b_2$ intersects the interior of $Q.$ 
As the diagonal $b_1b_2$ exists, there must be  a linear function $\Lambda,$ such that
$\left.\Lambda\right|_P\geq f,$ $\Lambda(b_{1,2})=f(b_{1,2}). $ 
As $L$ majorizes $f,$ we have that $\Lambda(b_{1,2})=f(b_{1,2})\leq L(b_{1,2}).$ Thus, the real linear function $d$ defined by $d(z)=\Lambda(z)-L(z)$  satisfies 
 $d(b_{1,2})\leq 0.$ 
Since $b_1b_2$ intersects the interior of $Q,$  there are two vertices $a_k$ and $a_j$, $1\leq k< j\leq 4,$ which are separated by the line containing the diagonal $b_1b_2$.
Since $\Lambda$ majorizes $f$, we obtain that 
$\Lambda(a_{k,j})\geq f(a_{k,j})= L(a_{k,j})$ and 
 $d(a_{k,j})\geq 0.$ Such a behaviour of the sign of a linear function $d$ is possible if and only if $d\equiv 0.$ Therefore, there may be only one linear function majorizing $f$ and coinciding with it at $b_{1,2},$ which contradicts the existence of the diagonal $b_1b_2.$
  
\end{proof}

The converse to  Lemma \ref{oneside} is also true, and we need the following lemma to prove this.

\begin{lemma}\label{sides}
Let a weighted chord diagram have a chord $a_1a_2$ with some positive weight. Let $f$ be the balanced weight corresponding to the diagram.
Then there are two distinct vertices $a_3$ and $a_4,$ which lie on different sides with respect to $a_1a_2,$ such that there exist distinct linear functions $L$ and $\Lambda$ majorizing $f$ and satisfying the relations 
\begin{equation}\label{maj0}
 \begin{array}{ll}
  \left.L\right|_{P}\geq f,\,\,  \left.L\right|_{P}(a_{1,2,3})=f(a_{1,2,3}),\\ \\
  \left.\Lambda\right|_{P}\geq f,\,\, \left.\Lambda\right|_{P}(a_{1,2,4})=f(a_{1,2,4}).
  \end{array}
\end{equation}
\end{lemma}
\begin{proof}
Indeed, there must exist distinct vertices $a_3,$ $a_4$ and 
distinct linear functions $L,$ $\Lambda$ satisfying relations (\ref{maj0}), because otherwise the chord does not have a positive weight. 
Let us assume now that $a_3$ and $a_4$ lie on one side with respect to  the line segment $a_1a_2.$ The relations (\ref{maj0}) imply that 
\[\begin{array}{ll}
  L(a_3)=f(a_3)\leq \Lambda(a_3),\\ \\
 \Lambda(a_4)=f(a_4)\leq L(a_4),
 \end{array}\]
and thus, the linear function $d(z)=\Lambda(z)-L(z)$ satisfies the inequalities
$d(a_3)\geq 0$ and $d(a_4)\leq 0.$ In addition we have that 
$d(a_{1,2})=0.$ As the points $a_3$ and $a_4$ lie on one side with respect to the chord $a_1a_2,$ the linear function $d$ has to vanish identically, which contradicts the existence of the chord $a_1a_2.$
\end{proof}

\begin{lemma}
A balanced weight corresponding to a weighted chord diagram with an incomplete triangulation is degenerate.
\end{lemma}
\begin{proof}
Let $f$ be a balanced weight which corresponds to the  weighted chord diagram~$P.$
If $P$ has no diagonals, there must be a linear function whose restriction to $P$ is identically equal to $f,$ and thus $f$ is degenerate. Assume the contrary.
We denote by $a_1$ and $a_2$ the vertices of $P$ such that the values $f(a_1)$ and $f(a_2)$ are the biggest and second biggest values of $f.$  Let $a_3$ and $a_4$ be the argument of the maximum of $f$ among the vertices to the left and to the right from the line segment $a_1a_2.$ Linear functions $L_1$ and $L_2$, which coincide with $f$ at $a_{1,2,3}$ and $a_{1,2,4}$ respectively, majorize $f.$ By assumption, $L_1$ and $L_2$ are not identically equal, and thus, the chord $a_1a_2$ must have a positive weight, and we arrive at a contradiction.

Assume now that the chord diagram has a chord $a_1a_2$ with a positive weight. Then there are a vertex $a_3$ and a linear function $L,$ such that $\left.L\right|_{P}\geq f,$  $\left.L\right|_{P}(a_{1,2,3})=f(a_{1,2,3}).$ The line segment $a_1a_3$ divides $P$ into two parts. Denote by $b_1,\dots,b_m$ the vertices of the part of $P$ which does not contain $a_2$. We define $b=\mathrm{arg max}_{1\leq k\leq m} f(b_k)$ and construct the linear function $L_1$ which coincides with $f$ at vertices $a_1,$ $a_3$ and $b.$ The function $L_1$ majorizes $f.$ If $L_1\equiv L,$ $f$ is degenerate. If $L_1\not\equiv L,$ the chord  $a_1a_3$ has a positive weight. Using Lemma~\ref{sides} we continue this construction of the chords until we either discover degeneracy of $f$ or obtain a complete triangulation of $P.$

\end{proof}
	We summarize the results stated previously in the following theorem.
	\begin{theorem}\label{thm1}
	The balanced weights defined on an $(n+1)-$gon $P$ constitute a vector space isomorphic to $\mathbb{R}^{n-2}.$ The degenerate balanced weights form the Stasheff fan $\Sigma_n.$ There is a one-to-one correspondence between the degenerate balanced weights and the weighted chord diagrams based on $P$ with an incomplete triangulation.
	\end{theorem}

	\end{section}

\begin{section}{Quadratic differentials}\label{sqd}

A meromorphic quadratic differential $q$ on a Riemann sphere $S$ is a meromorphic section of the symmetric square of the complexified cotangent bundle over $S$.
It is represented as $q(z)dz^2$ in the local parameter $z$ by a meromorphic 
 function $q(z)$ on $S$ together with the following transition rule 
\[
q^*(\zeta)=q(z(\zeta))\left(\frac{dz}{d\zeta}\right)^2,
\]
in the common neigbourhood of the parameters $z$ and $\zeta$,
where $q^*$is the same quadratic differential in terms of the local parameter $\zeta$.

A $\it{horizontal}$ (respectively, $\it{vertical}$) trajectory of quadratic differential is a maximal curve along which the inequality $q(z)\,dz^2> 0$, (respectively, $ q(z)\,dz^2< 0 $) holds. If the endpoint of a trajectory is a zero or a simple pole of $q$, such trajectory is called $\it{critical}$.

The zeros and poles of a quadratic differential are the $\it{critical}$ points. All non-critical points are $\it{regular}$. In a neighborhood of a regular point horizontal and vertical trajectories are just straight horizontal and respectively vertical lines. The trajectory structure about the critical points is  well-known, see e.g.,  \cite{Jenkins1958, Strebel1984,  Vasiliev2002}. Description of the global structure of a quadratic differentials is much more difficult.

We are interested in the following family of quadratic differentials:
		\begin{equation}\label{fam}
		\frac{z^k+a_1 z^{k-1}+\dots+a_0 }{z}\,dz^2
		\end{equation}
		where $k\ge2$, $z\in {\mathbb{C}}$, $a_j\in\mathbb{C},$ $0\leq j\leq k-2.$

Let  $q(z)dz^2$ be a member of the family (\ref{fam}). 
It has
 $k$ zeros (counting multiplicity), a simple pole at the origin and 
a pole $p$ of order $m=k+3$ at infinity. Observe that unlike the versal deformation of a polynomial quadratic differential the coefficient $a_1$ in \eqref{fam} is not 
necesserely vanishing, because the simple pole at the origin prevents to perform affine coordinate change.

Let us have a look at trajectory structure of $q(z)dz^2$ about infinity. If the infinite pole $p$ is of order $m$, then it is possible to find a neighbourhood $U$ of $p$, such that any trajectory ray entering $U$ stays in $U$. In this neighbourhood one can define $m-2$ so-called $\it{principal}$ $\it{directions}$, such that
 the directions divide $U$ into $m-2$ sectors of angles $\frac{2 \pi}{m-2}$; and
 any trajectory ray that enters $U$ tends to $p$ in one of the directions.

\begin{example}
Quadratic differential of form $\frac{z^2-1}{z}\, dz^2$ has three principal directions at infinity. Figure \ref{pole5} illustrates the trajectory structure about infinity in this case.

\begin{figure}[h!]
 \centering
    \includegraphics[width=0.3\textwidth]{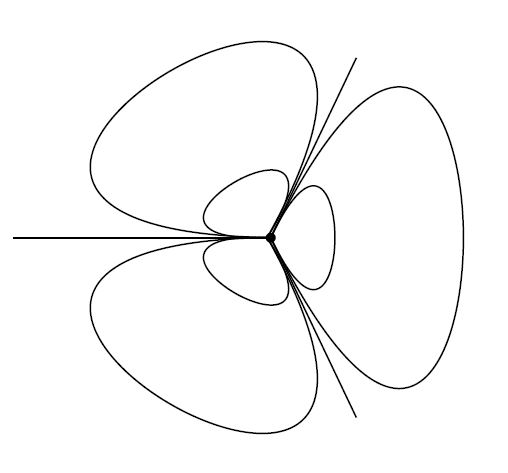}
 \caption{Trajectory structure near a pole of order 5}\label{pole5}
\end{figure}

\end{example}
We denote by $\Phi_q$ the union of all critical trajectories of $q(z)\,dz^2$. Then $S\setminus \bar{\Phi}$ splits  $S$ into $\it{strip}$ and $\it{ending}$ domains. Strip and ending domains in the trajectory structure of $q(z)\,dz^2$ are simply connected domains, which can be mapped conformally by $w=\int \sqrt{q(z)}\, dz$ onto a strip $\{a<\Im w< b\}$ and a halfplane respectively. These domains are swept 
out by the trajectories starting and ending at infinity; the critical points of $q(z)\,dz^2$ belong to their boundaries.

\end{section}

\section{Graph representation of quadratic differentials}\label{graphs}
In this section we establish the one-to-one correspondence between weighted graphs and quadratic differentials of the form (\ref{fam}).
\subsection{Assigning admissible graphs to quadratic differentials}
We describe an algorithm of assigning a pair of graphs to a quadratic differential. 

Suppose we are given a quadratic differential $q(z)\,dz^2$ from the family (\ref{fam}). It has $k$ zeros, a simple pole at the origin and a pole of order $k+3$ at infinity. Thus, the horizontal and vertical trajectory structures of $q(z)\,dz^2$ about the infinite pole have $k+1$ principal directions each. 

We construct graphs $G_h$ and $G_v$ which represent the horizontal and  vertical trajectory structure of $q(z)\,dz^2$ respectively.

The graph $G_h$ (respectively, $G_v$) contains $k+1$ vertices and edges, which form a regular convex $(k+1)-$gon, denoted by $\Pi_h$ (respectively, $\Pi_v$).
In addition, each graph has a vertex $O$ placed in the center if the interior of 
the $(k+1)-$gon.
 The vertices of 
$\Pi_h$ (respectively, $\Pi_v$)
represent the principal directions in which the critical trajectories tend to the infinite pole. The vertex $O$ represents the finite pole.

The quadratic differential $q(z)\,dz^2$ is characterized uniquely by the domains of the trajectory structure, i.e., strip and half-plane domains.
The half-plane domains are represented by the edges of
$\Pi_h$ (respectively, $\Pi_v$).
In order to mark the strip domains of horizontal (respectively, vertical)
trajectory structure of  $q(z)\,dz^2$ on the $G_h$ (respectively $G_v$) we construct additional edges. 
 Suppose we have a strip domain $S$, which is to be marked on the graph $G_h$ or $G_v$.
 As any strip domain it is swept out by trajectories whose ends approach infinity in certain principal directions. Suppose these two principal directions are represented by the vertices $a$ and
 $b$ of $\Pi_h$ or $\Pi_v$. Observe that $a$ and $b$ may coincide.
 If the strip domain $S$ does not have the finite pole on its boundary,  
 we mark it with an edge joining the vertices $a$ and $b$. 
If $S$ has the finite pole on its boundary, we mark it with
 an edge joining the pole vertex $O$ with $a$ and an edge joining  $O$ with $b$. 
 Finally to each edge representing $S$ we assign a weight $w_S$ which is equal the the width of $S$ in the metric associated with the quadratic differential $q(z)\,dz^2$.

This way we mark all the strip domains of the horizontal (respectively vertical) trajectory structure of $q(z)\,dz^2$ on $G_h$ (respectively $G_v$) so that the edges of $G_h$ (respectively $G_v$) intersect only at vertices.

\begin{example}\label{exx}
The quadratic differential $\frac{z^3-1}{z}\,dz^2$ has 3 simple zeros, a simple pole at the origin, and the infinite pole of order 6. The horizontal and vertical trajectory structures have 4 principal directions at infinity  each. Figure~\ref{gr} illustrates the graphs $G_h$ and $G_v$.
\begin{figure}
\includegraphics[scale=0.2]{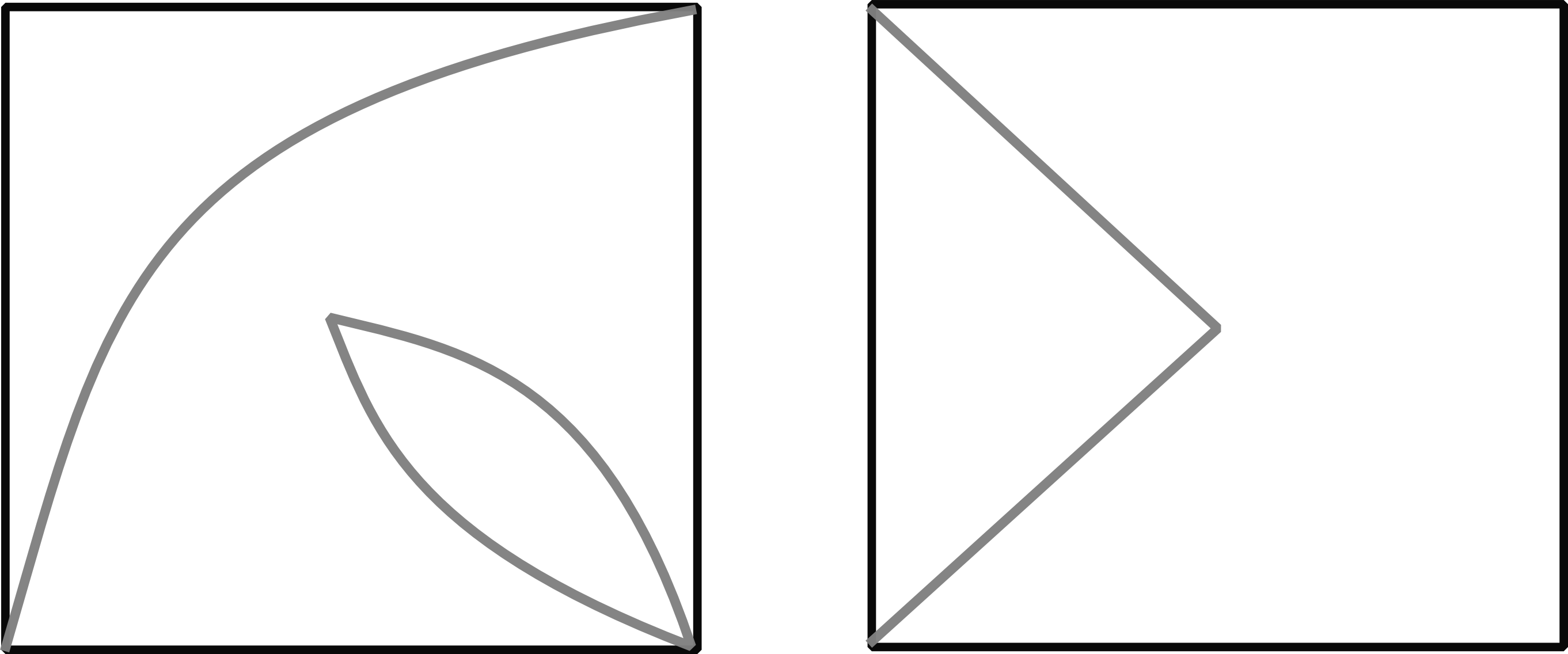}\caption{Graphs $G_h$ and $G_v$ for $\frac{z^3-1}{z}\,dz^2$. }\label{gr}
\end{figure}
\end{example}

\subsection{Admissible graphs}
Let us describe the graphs which may represent a quadratic differential.
We call a graph $\it{admissible}$ if we can associate a horizontal or vertical trajectory structure of a
quadratic differential to it. 
An admissible graph $\Gamma$ contains $n$, $n\geq 3,$ vertices and edges which constitute a regular convex polygon $\Pi_{n}$. In addition, 
$\Gamma$ has vertex $O$ at the center of the interior of $\Pi_{n}$.
There are two edges connecting the vertex $O$ with some adjacent vertices $a$ and $b$ of $\Pi_n$. The vertices $a$ and $b$ may coincide. and in this case, we treat them as two adjacent vertices with one and the same support.
The edges of an admissible graph intersect only at vertices.

\subsection{Assigning a quadratic differential to a pair of admissible graphs }
 
Here we describe an algorithm of assigning the trajectory structure of a quadratic differential to a pair $(G_h, G_v)$ of admissible graphs with $k+2$ vertices.
Let us start with merging $G_h$ and $G_v$ into one and the same graph $G$, assigning to $G_h$ and $G_v$ different colours. 
Then, we place $G_h$ over $G_v$ in such a way, that the vertex $O$  of 
$G_h$ is right above the  vertex $O$ of $G_v$, and the vertices of the polygons
 $\Pi_h$ and $\Pi_v$ are interlacing. Furthermore, we erase the edges forming the polygons $\Pi_h$ and $\Pi_v$, and join the $2(k+1)$ interlacing vertices with edges, so that a regular convex $2(k+1)-$gon $\Pi$ is formed. Finally, we merge what is left of $G_h$ and $G_v$ with the 
$2(k+1)-$gon $\Pi$ into a new graph $G$. Note, that the edges of $G$ may intersect not only at vertices.

\begin{example}\label{exxx}
The graphs $G_h$ and $G_v$ from example \ref{exx} are admissible. The corresponding graph $G$ is shown in Figure~\ref{g}.
\begin{figure}
\includegraphics[scale=0.3]{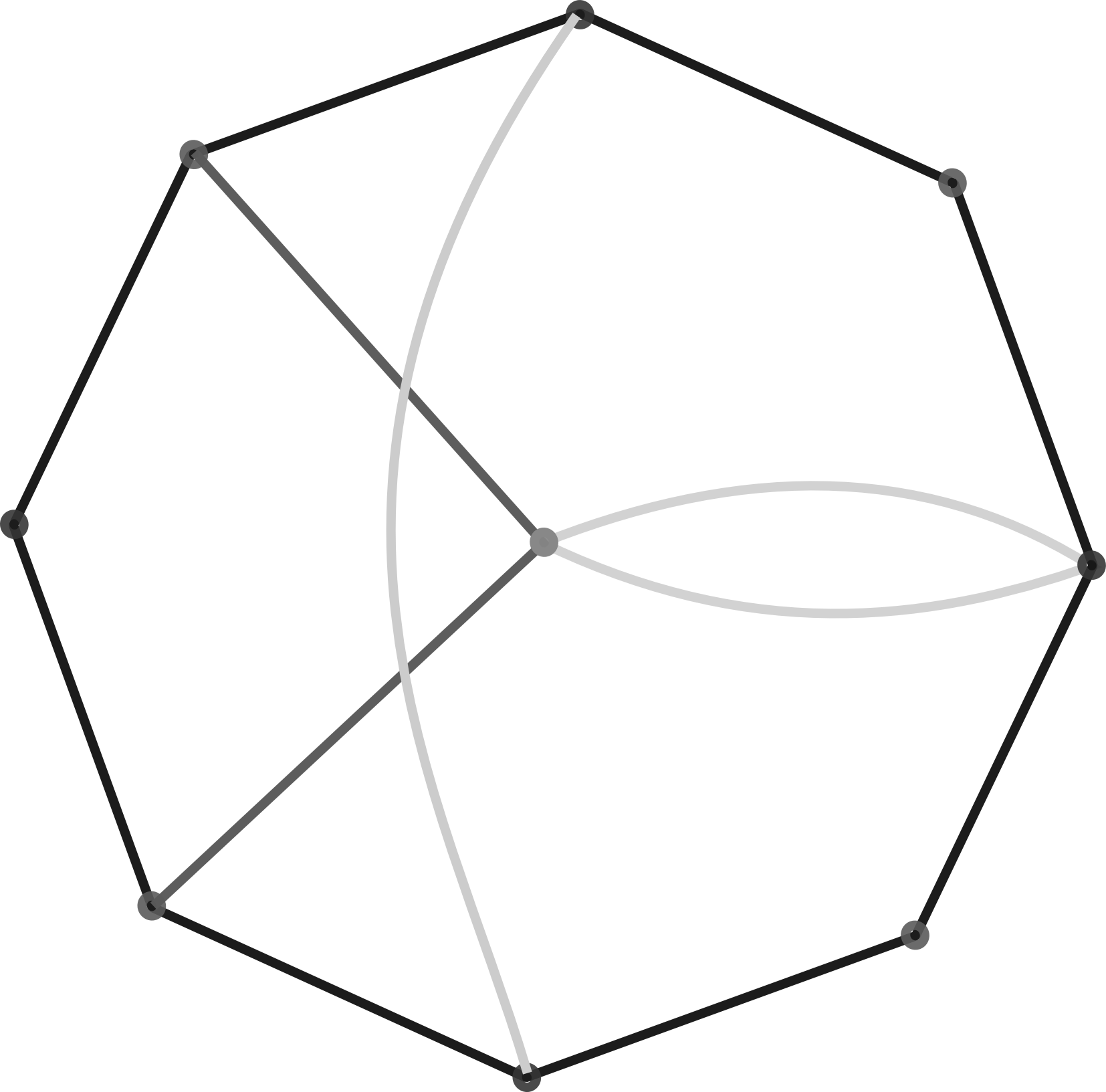}\caption{Graph $G$ for Examples \ref{exx} and \ref{exxx}.}\label{g}
\end{figure}
\end{example}

Further, let us describe an algorithm of the construction of
an extended graph $G_{ext}$. The constructed edges   represent further pieces of critical trajectories of a quadratic differential.
Hence, we specify the correspondence between $G_{ext}$ and the trajectory structure of a quadratic differential.

\begin{remark}\label{doublesupp}
For the construction we need the following rule: 
if the graph $G$ has a double edge with ends at the vertex $O$ and a vertex $b$ of the polygon $\Pi$, then $b$ counts as two vertices with one and the same support.
\end{remark}

\subsection{Algorithm of construction of $G_{ext}$}
By admissibility of $G_h$ and $G_v$ the interior of the polygon $\Pi$ is divided by edges of $G$ into at least four connected components. 
 Pick  a point
 \label{st3}  in each connected component. We call these points {\it component centers}. The component centers represent points of intersection of critical trajectories of a quadratic differential.
 \label{firstchords} 
If the boundary of a component contains vertices of the polygon $\Pi$, then connect the component center with these vertices by line segments.
 \label{secondchords} Whenever boundaries of two connected components share a piece of an edge of $G$, connect the component centres by a line segment.

 \label{st6} 
After the completion of previous steps the interior of the polygon $\Pi$ is divided into triangles and quadrilaterals.
Whenever the boundary of a quadrilateral contains the  vertex $O$ and two pieces of the edges of the same colour, we construct a line segment connecting $O$ with the non-adjacent vertex of the quadrilateral.
 Such a line segment represents a piece of a critical trajectory. 
 This completes the construction of $G_{ext}$

 The edges of $G_{ext}$ divide the interior of the polygon $\Pi$ into the following domains:
\begin{itemize}
\item[(a)] Triangles having a side of $\Pi$ as a side;
\item[(b)] Triangles having only one vertex of $\Pi$ as a vertex. A piece of an edge of $G_h$ or $G_v$ constitutes one of the triangles sides;
\item[(c)] Quadrilateral  having 2 pieces of edges of $G$ of different color as adjacent sides;
\item[(d)] Triangle whose boundary contains the vertex $O$ and a piece of an edge of $G$.
\end{itemize}
Each triangle of type (a) can be identified with a
quadrant. 
The boundary of a triangle of type (b) contains a piece of an edge of $G_h$ or $G_v$ of weight $v$. Then the triangle is identified with a quarter of the strip $\{a<\Im w< b\}$, where $b-a=v$. 
The boundary of a quadrilateral  of type (c) contains a piece of
edge of $G_h$ of weight $v$, and a piece of
edge of $G_v$ of weight $u$. Then the quadrilateral  is identified with the rectangle with sides of length $v$ and $u$.

The union of quadrilaterals and triangles with the vertex $O$ at the boundary is identified with a rectangle, which is a part of a strip. The width of the rectangle is given by the weight of one of the coloured sides of the quadrilaterals and triangles.

Recall that each strip or ending domain of the trajectory structure of a quadratic differential is mapped conformally onto an infinite strip or a half-plane. The identification described above establishes the correspondence between the domains formed by $G_{ext}$ and the domains of the trajectory structure of a quadratic differential.

\begin{example}
Figure~\ref{gext} illustrates 
two copies of the graph $G_{ext}$ corresponding to the graph $G$ from Example~\ref{exxx}.  The dashed lines represent the critical trajectories. The right-hand side copy has the shadowed region representing a strip domain.
\begin{figure}
\includegraphics[scale=0.3]{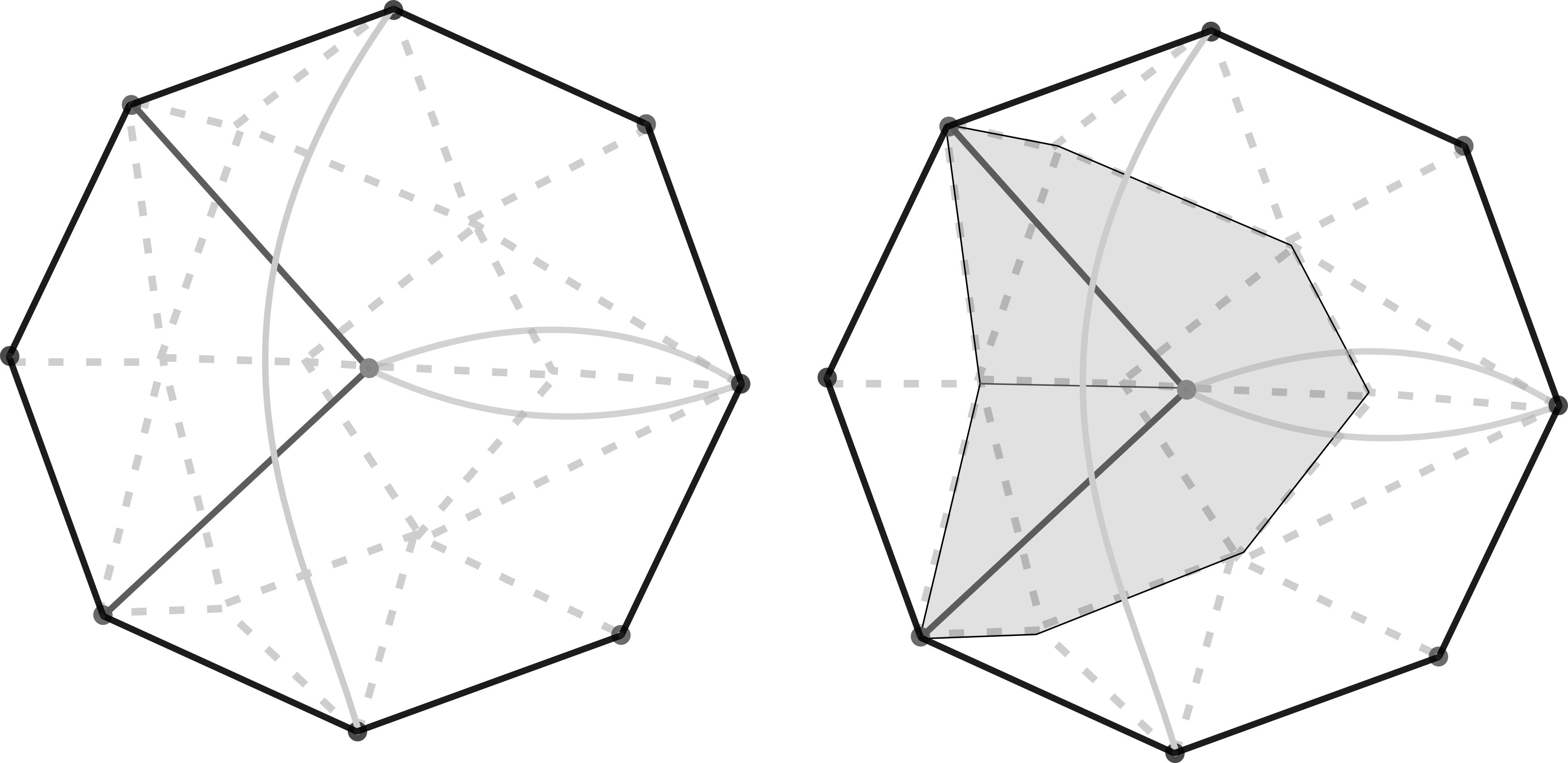}\caption{$G_{ext}$ for the graph $G$ from Example \ref{exxx}. }\label{gext}
\end{figure}
\end{example}

The position and weights of the coloured edges of $G_{ext}$ define uniquely a quadratic differential representing the original pair of graphs $(G_h,G_v).$ The position of coloured edges defines the relative position of the strip domains, while the weights fix their width in the natural metric. More precisely, the mapping $w=\int\sqrt{q(z)}dz$ maps the complex plane
onto a Riemann surface branched at the images of the zeros of $q$ and the regular trajectories are mapped onto the horizontal straight lines  in the $w$-plane.

	\subsection{Correspondence between triangulation and the short trajectories.}

	We established the one-to-one correspondence between quadratic differentials of the form~(\ref{fam}) and pairs of admissible graphs in the previous sections.
	In what follows, let us specify the graphs which represent quadratic differentials with  short trajectories.
	
	We describe how an admissible graph $G_h$ (respectively, $G_v$) gives rise to a weighted chord diagram. Suppose that the graph $G_h$ (respectively, $G_v$) has $k+2$ vertices, and let the vertex $O$ be connected with the vertices $a$ and $b$ by edges $u$ and $v$. We erase the vertex $O$ and replace $u$ and $v$ by a single  edge joining $a$ and $b$. If $a$ and $b$ have the same support, we disunite it, so that  $a$ and $b$ become two separate adjacent vertices. The resulting graph $\Gamma_h$ (respectively, $\Gamma_v$) has $n$ vertices, where 
$n=k$ 	if $a$ and $b$ originally had different supports,
	and  $n=k+1$ if $a$ and $b$ originally had coinciding supports.
	 The graph $\Gamma_h$
(respectively, $\Gamma_v$) is isomorphic to a regular convex $n-$gon with weighted diagonals. 
 This convex realisation of $\Gamma_h$ (respectively, $\Gamma_v$) is exactly the desired weighted chord diagram.
Analogously, a pair of weighted chord diagrams with an appropriate number of vertices gives rise to a pair of admissible graphs.

The diagonals of $\Gamma_h$ (respectively, $\Gamma_v$) generate  triangulation of the $n-$gon. The following lemma provides characterization of quadratic differentials with short trajectories. 
\begin{lemma}\label{kor}
The trajectory structure represented by $G_h$ (respectively, $G_v$) has a short trajectory joining two zeros if and only if the triangulation of the corresponding $n-$gon is incomplete.
\end{lemma}

\begin{example}
Figure \ref{wd} shows the weighted diagrams $\Gamma_h$ and $\Gamma_v$ associated with the graphs $G_h$ and $G_v$ from Example \ref{exx}. As we can see, the triangulations of $\Gamma_h$ and $\Gamma_v$ are incomplete. The corresponding trajectory structures have short trajectories.
\begin{figure}\label{wd}
\includegraphics[scale=0.3]{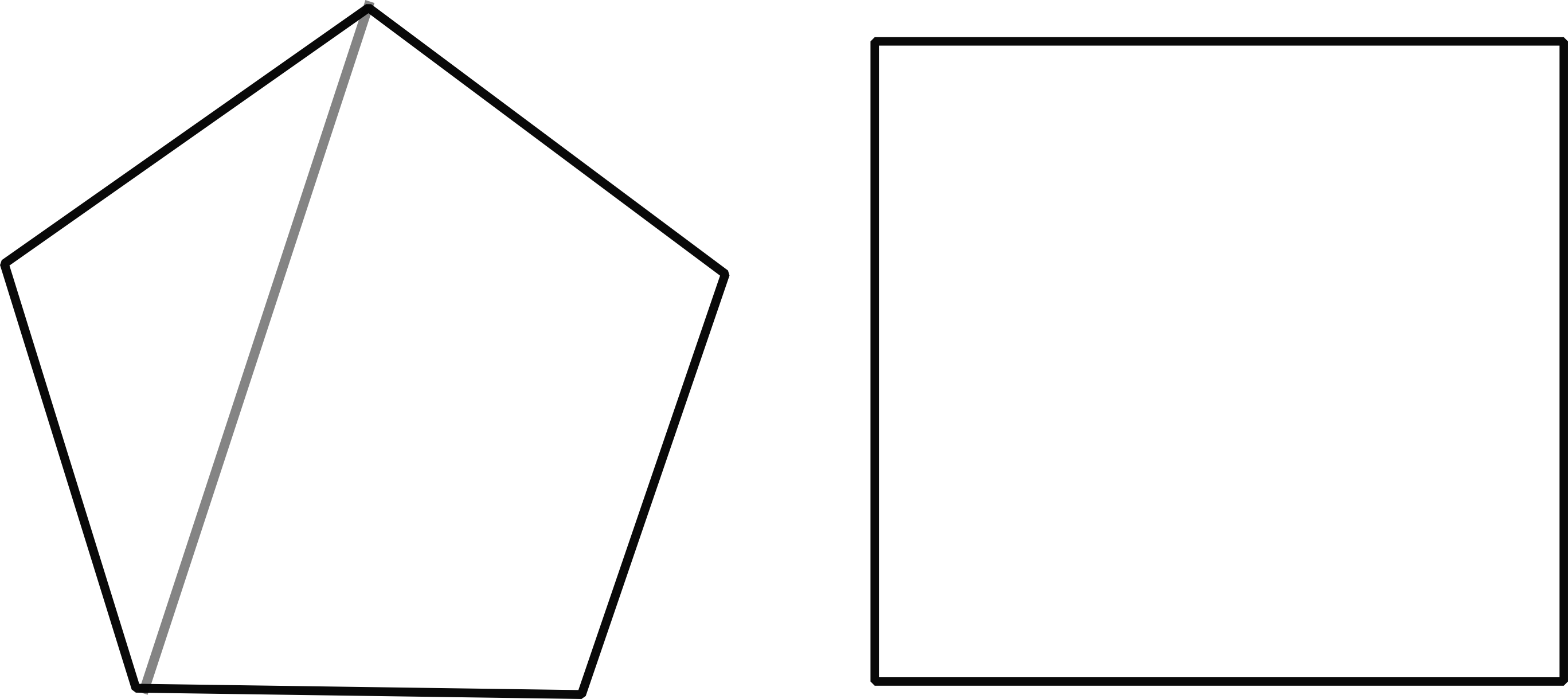}\caption{$\Gamma_h$ and $\Gamma_v$.}
\end{figure}
\end{example}

\begin{subsection}{Parametric space}\label{main}

Our goal is to characterize the set of parameters $S$ in the parameter space $\Lambda\cong \mathbb{R}^{2k}$, for which the corresponding quadratic differential has a short trajectory joining two zeros. The set $S$ naturally splits into the horizontal and the vertical components $S_h$ and $S_v$.

\begin{theorem}
The horizontal and vertical components of the set $S$ have the following form:
\[S_h=\left(\Sigma_k\times \mathbb{R}^{k+2}\right)\cup \left(\Sigma_{k+1}\times \mathbb{R}^{k+1}\right),
\]

\[ S_v=\left(\mathbb{R}^{k+2}\times\Sigma_k\right) \cup \left(\mathbb{R}^{k+1}\times \Sigma_{k+1}\right).
\]
\end{theorem}
\begin{proof}
By Theorem (\ref{thm1}) and Lemma (\ref{kor}) a quadratic differential of the form (\ref{fam}) with a short trajectory can be identified with a point in the fan $\Sigma_k$ or $\Sigma_{k+1}.$ The set $S$ has codimention 1, which leads us to the statement of the theorem.

\end{proof}

\begin{remark}
Quadratic differentials of the form (\ref{fam}) contain a subfamily of quadratic differentials
\begin{equation}\label{degen}
(z^{k-1}+a^{*}_{k-2}z^{k-2}+\dots+a^{*}_0)\,dz^2,
\end{equation}
where $a^{*}_{k-2},\dots,a^{*}_{0}$ are complex parameters.
By Baryshnikov's result \cite{Baryshnikov2001} the bifurcation diagram $S^{*}$ of this family consists of components $S^{*}_v=\mathbb{R}^{k-2}\times \Sigma_{k}$ and $S^{*}_h=\Sigma_{k}\times\mathbb{R}^{k-2}$ in the parameter space $\mathbb{R}^{2(k-2)}.$ Therefore, $S^{*}$ is exactly the subset of $S$ corresponding to quadratic differentials of the form (\ref{degen}).
\end{remark} 
\end{subsection}

\end{document}